\documentclass[smallextended]{svjour3}
\usepackage{amsmath,amsxtra,amssymb,latexsym, amscd}%,refcheck}
\usepackage[utf8]{inputenc}
\usepackage[mathscr]{eucal}
\usepackage[english]{babel}
\usepackage[active]{srcltx}
\usepackage{color}

\newcommand{\R}{\mathbb{R}}

\def\EES{{\accent"5E e}\kern-.5em\raise.8ex\hbox{\char'23 }}
\def\ow{o\kern-.42em\raise.82ex\hbox{
   \vrule width .12em height .0ex depth .075ex \kern-0.16em \char'56}\kern-.07em}
\def\OW{o\kern-.460em\raise1.36ex\hbox{
\vrule width .13em height .0ex depth .075ex \kern-0.16em
\char'56}\kern-.07em}
\def\DD{D\kern-.7em\raise0.4ex\hbox{\char '55}\kern.33em}

\title{H\"{o}lder-Type Global Error Bounds for Non-degenerate Polynomial Systems}

\author{S\~i Ti\d{\^e}p \DD INH$^\dagger$ \and 
		Huy Vui H\`A$^\dagger$\and 
		Ti\EES n S\OW n PH\d{A}M$^\ddagger$}

\institute{S\~i Ti\d{\^e}p \DD INH$^\dagger$ \at
              Institute of Mathematics, VAST, 18, Hoang Quoc Viet Road, Cau Giay District 10307, Hanoi, Vietnam \\
              \email{dstiep@math.ac.vn}
           \and
           Huy Vui H\`A$^\dagger$ \at
              Institute of Mathematics, VAST, 18, Hoang Quoc Viet Road, Cau Giay District 10307, Hanoi, Vietnam \\
              \email{hhvui@math.ac.vn}
           \and   
           Ti\EES n S\OW n PH\d{A}M$^\ddagger$ \at
			  Department of Mathematics, University of Dalat, 1 Phu Dong Thien Vuong, Dalat, Vietnam\\
			  \email{sonpt@dlu.edu.vn} 
}

\authorrunning{S. T. \DD INH, H. V. H\`A, T. S. PH\d{A}M}
\date{\today}

\begin{document}

\maketitle

\begin{abstract}
Let $F := (f_1, \ldots, f_p)  \colon {\Bbb R}^n \to {\Bbb R}^p$ be a polynomial map, and suppose that $S := \{x \in {\Bbb R}^n \ : \ f_i(x) \le 0, i = 1, \ldots, p\} \ne \emptyset.$ Let $d := \max_{i = 1, \ldots, p} \deg f_i$ and $\mathcal{H}(d, n, p) := d(6d - 3)^{n + p - 1}.$ Under the assumption that the map $F \colon {\Bbb R}^n \rightarrow {\Bbb R}^p$ is convenient and non-degenerate at infinity, we show that there exists a constant $c > 0$ such that
the following so-called {\em H\"{o}lder-type global error bound result} holds
$$c d(x,S) \le [f(x)]_+^{\frac{2}{\mathcal{H}(2d, n, p)}} + [f(x)]_+ \quad \textrm{ for all } \quad x \in \mathbb{R}^n,$$
where $d(x, S)$ denotes the Euclidean distance between $x$ and $S,$ $f(x) := \max_{i = 1, \ldots, p} f_i(x),$ and $[f(x)]_+ := \max \{f(x), 0 \}.$ The class of polynomial maps (with fixed Newton polyhedra), which are non-degenerate at infinity, is generic in the sense that it is an open and  dense semi-algebraic set. Therefore, H\"{o}lder-type global error bounds hold for a large class of polynomial maps, which can be recognized relatively easily from their combinatoric data. This follows up the result on a Frank-Wolfe type theorem for non-degenerate polynomial programs in \cite{Dinh2013}.

\end{abstract}

\section{Introduction}

Let $f_1, \ldots, f_p  \colon {\Bbb R}^n \to {\Bbb R}$ be polynomial functions, and suppose that the set
$$S := \{x \in {\Bbb R}^n \ : \ f_1(x) \le 0, \ldots, f_p(x) \le 0\}$$
is nonempty. Let $f(x) := \max_{i = 1, \ldots, p} f_i(x).$ Then $S = \{x \in {\Bbb R}^n \ : \ f(x) \le 0\}.$
We are interested in the question of  whether one can use the residual (constraint violation) at a point $x \in {\Bbb R}^n$   to bound the distance from $x$ to the set $S.$ More precisely, we study if there exist some positive constants $c, \alpha,$ and $\beta$ such that
\begin{equation} \label{HolderInequality}
c d(x,S) \le [f(x)]_+^{\alpha} + [f(x)]_+^\beta \quad \textrm{ for all } \quad x \in \mathbb{R}^n,
\end{equation}
where $d(x, S)$ denotes the Euclidean distance between $x$ and the set $S$ and $[f(x)]_+ := \max \{f(x), 0 \}.$ We say that a {\em H\"older-type global error bound} holds for the set $S$ if the inequality (\ref{HolderInequality}) holds.

The study of error bounds has grown significantly and has found many important applications. In particular, it has been used sensitivity analysis for various problems of mathematical programming (for examples, the variational inequality, the linear and nonlinear complementarity problem, and the 0-1 integer feasibility problem). It has also been used as termination criteria for iterative decent algorithms. For a  comprehensive, state of the art survey of the extensive theory and rich  applications  of error  bounds, we refer the readers to the survey of Pang \cite{Pang1997} and the references cited therein.

The first error bound result is due to Hoffman \cite{Hoffman1952}. His  result deals with  the case where the polynomials $f_1, \ldots, f_p$ are affine and states that the inequality (\ref{HolderInequality})  holds with the exponents $\alpha = \beta  = 1.$
After the work of Hoffman, many people have devoted themselves to the study of global error bound; see, for example,
\cite{Auslender1988}, \cite{Klatte1998}, \cite{Klatte1999}, \cite{Lewis1998}, \cite{Mangasarian1985}, \cite{Ngai2005}, \cite{Robinson1975}.

%Under the convexity assumption of the polynomials $f_i,$ global H\"{o}lder-type error bounds have been shown in \cite{Li2010, Li2012, Luo1994-1, Luo1994-2, Luo2000-1}

In general, without the assumption of convexity, H\"{o}lder-type global error bounds are highly unlikely to hold.
When the constrained set $S$ defined by some affine functions and a {\em single quadratic} polynomial, Luo and Sturm \cite{Luo2000-1} showed that the H\"older-type global error bound (\ref{HolderInequality}) holds with the exponents $\alpha= \displaystyle\frac{1}{2}$ and $\beta = 1.$ In particular, a H\"{o}lder-type global error bound (with some {\em unknown} fractional exponents $\alpha$ and $\beta$) was obtained very recently by H\`a \cite{HaHV2013} for a nonlinear inequality defined by a {\em single} polynomial (i.e., in the case where $p = 1$),  which is convenient  and non-degenerate at infinity (see \cite{Kouchnirenko1976} and Section \ref{SectionPreliminary} for precise definitions).

This paper will deal mainly with a class of polynomial maps, which are defined by combinatorial data and are called {\em non-degenerate at infinity.} This notion is an adaptation in the real setting of the notion of non-degeneracy in the complex setting given by \cite{Khovanskii1978}, \cite{Kouchnirenko1976}. In both real and complex contexts, 
the class of polynomial maps (with fixed Newton polyhedra), which are non-degenerate at infinity, is generic in the sense that it is an open and dense set.

For any positive integers $d, n,$ and $p,$ let
$$\mathcal{H}(d, n, p) := d(6d - 3)^{n + p - 1}.$$
With the definitions in the next section, the main contribution of this  paper is the following H\"{o}lder-type global error bound with explicit exponents. 

\begin{theorem} {\rm (Compare \cite[Theorem C]{HaHV2013})} \label{HolderTypeTheorem}
Let $F := (f_1, \ldots, f_p) \colon {\Bbb R}^n \rightarrow {\Bbb R}^p, 1 \le p \le n,$ be a polynomial map. Suppose that $F$ is convenient and non-degenerate at infinity. Let $f(x) := \max_{i = 1, \ldots, p} f_i(x)$ and $S := \{x \in {\Bbb
R}^n \ : \ f(x) \le 0\} \ne \emptyset.$ Then there exists a constant $c>0$ such that
\begin{equation}\label{GlobalErrorBound}c d(x,S) \le [f(x)]_+^{\frac{2}{\mathcal{H}(2d, n, p)}} + [f(x)]_+ \quad \textrm{ for all } \quad x \in \mathbb{R}^n,\end{equation}
where $d:=\max_{i = 1, \ldots, p}\deg f_i.$
\end{theorem}

Our result extends the result of \cite{HaHV2013}, which studies the case $p=1$. We also give estimations of the exponents $\alpha,\beta$ in (\ref{HolderInequality}), which has not been done in \cite{HaHV2013}.
In the case $p=1$, $F=f_1$, the existence of H\"older-type global error bounds follows easily from the existence of the following global \L ojasiewicz-type inequality
\begin{equation} \label{LojasiewiczInequality}
c d(x,Z) \le |F(x)|^{\alpha} + |F(x)|^\beta \quad \textrm{ for all } \quad x \in \mathbb{R}^n,
\end{equation} 
where $Z=F^{-1}(0)$. For a system of polynomials which is non-degenerate at infinity, the existence of (\ref{LojasiewiczInequality}) has been proved by \cite{Dinh2014}; however, the existence of H\"older-type global error bounds does not follow directly from (\ref{LojasiewiczInequality}). Our method is actually different from \cite{HaHV2013} at the crucial point that we use only the Curve Selection Lemma at infinity (see Lemma \ref{Lemma1}) as a tool.
The reader may find other global versions of this inequality in papers \cite{HaHV2010}, \cite{Hormander1958}, \cite{Ji1992}.

%It is stated by the authors of \cite{Luo1994-2} in their concluding remarks that ``we have not been able to obtain explicit formulas for the multiplier or the exponent in the error bound. We feel that such formulas would be useful for computational and other purposes." 

It is worth notice that error bound results with {\em explicit} exponents are indeed important for both theory and applications since they can be used, e.g., to establish explicit {\em convergence rates} of the proximal point algorithm as demonstrated in \cite{Borwein2014}, \cite{LM2012}, \cite{Li1995}. %We also refer the reader to \cite{Luo1996} for relevant discussions on other algorithms.

Note that we do not impose the condition of convexity on the polynomials $f_i,$ and their degrees can be arbitrary. 
Further, by genericity of the condition of non-degeneracy at infinity, the H\"older-type global error bounds hold for almost polynomial maps. 

The above H\"older-type global error bound result, together with the Frank-Wolfe type theorem in 
\cite{Dinh2013}, suggests that the class of polynomial maps, which are non-degenerate at infinity, may offer an appropriate domain on which the machinery of polynomial optimization works with full efficiency.

Let $f \colon {\Bbb R}^n \rightarrow {\Bbb R}$ be a continuous semi-algebraic function. Assume that $S := \{x \in {\Bbb R}^n \ : \ f(x) \le 0 \} \ne \emptyset.$ In order to obtain the main theorem, we have established some intermediate results which are of independent interest. Our principal tool is Curve Selection Lemma at infinity and the proof is closed to \cite{Dinh2013}, \cite{Dinh2014}. The sketch of the proof of the main result is as follows.

\begin{itemize}
\item First of all, we provide a necessary and sufficient condition for the existence of a H\"older-type global error bound for the set $S$ (Theorem~\ref{NSOAT}).

\item Secondly, we show that if $f$ satisfies the {\em Palais-Smale condition} at each non-negative value, then $f$ satisfies the above sufficient condition (Theorem~\ref{TheoremPalais-Smale-Global-Holderian}).

\item Thirdly, we show that if $f$ has a \textit{good asymptotic behavior} at infinity, then a H\"older-type global error bound (with the exponent $\beta = 1$) holds for the set $S$ (Theorem~\ref{GoodnessTheorem}).

\item Finally, we prove that if the map $(f_{1}, \ldots, f_{p}) \colon {\Bbb R}^n \rightarrow {\Bbb R}^p$  is convenient and non-degenerate at infinity, then $f(x) := \max_{i = 1, \ldots, p} f_i(x)$ has a good asymptotic behavior at infinity (Lemma~\ref{Lemma8}). Specifically, we determine explicitly the exponents $\alpha$ and $\beta$ in the H\"older-type global error bound (\ref{HolderInequality}) for such polynomial systems.
\end{itemize}

The paper is structured as follows. Section \ref{SectionPreliminary} presents some backgrounds in the field. A criterion for the existence of a H\"older-type global error bound is given in Section \ref{SectionErrorBound}. A relation between the Palais-Smale condition and the existence of H\"{o}lder-type global error bounds is given in Section \ref{SectionPalaisSmale}. In Section \ref{SectionGoodatinfinity}, we consider goodness at infinity.  The H\"older-type global error bound result (Theorem \ref{HolderTypeTheorem}) for convenient and non-degenerate polynomial maps will be proven in Section \ref{SectionHolder}. In Section \ref{Examples}, we give some illustrated examples for Theorem \ref{HolderTypeTheorem}. Section \ref{Applications} presents some applications.

\section{Preliminaries} \label{SectionPreliminary}

\subsection{Semi-algebraic geometry}
In this subsection, we recall some notions and results of semi-algebraic geometry, which can be found in \cite{Benedetti1991}, \cite{Bierstone1988}, \cite{Bochnak1998}, \cite{Dries1996}.

\begin{definition}{\rm
\begin{enumerate}
  \item[(i)] A subset of $\R^n$ is called {\em semi-algebraic} if it is a finite union of sets of the form
$$\{x \in\R^n \ : \ f_i(x) = 0, i = 1, \ldots, k; f_i(x) > 0, i = k + 1, \ldots, p\}$$
where all $f_{i}$ are polynomials.
 \item[(ii)]
Let $A \subset \Bbb{R}^n$ and $B \subset \Bbb{R}^p$ be semi-algebraic sets. A map $F \colon A \to B$ is said to be {\em semi-algebraic} if its graph
$$\{(x, y) \in A \times B \ : \ y = F(x)\}$$
is a semi-algebraic subset in $\Bbb{R}^n\times\Bbb{R}^p.$
\end{enumerate}
}\end{definition}

We list below some basic properties of semi-algebraic sets and functions.
\begin{enumerate}
\item[(i)] The class of semi-algebraic sets is closed under Boolean operators, taking Cartesian product, closure and interior.

\item[(ii)] A composition of semi-algebraic maps is a semi-algebraic map; the  image  and  preimage of a  semi-algebraic  set  under  a  semi-algebraic map are semi-algebraic sets;

\item[(iii)]  If $S$ is a semi-algebraic set, then the distance function $$d(\cdot, S) \colon {\Bbb R}^n \rightarrow {\Bbb R}, \quad x \mapsto d(x, S) := \inf \{\|x - a\| \ : \ a \in S \},$$
is also semi-algebraic.
\end{enumerate}

%A major fact concerning the class of semi-algebraic sets is its stability under linear projections (see, for example, \cite{Benedetti1991}, \cite{Bochnak1998}).

%\begin{theorem}[Tarski-Seidenberg Theorem] \label{TarskiSeidenbergTheorem}
%The image of a semi-algebraic set by a semi-algebraic map is semi-algebraic.
%Let $\pi(x_1, \ldots, x_n) = (x_1, \ldots, x_{n - 1})$ be the canonical projection from $\Bbb{R}^{n}$ onto $\Bbb{R}^{n - 1}.$ If $S$ is
%a semi-algebraic subset of $\Bbb{R}^{n},$ then so is $\pi(S)$ in $\Bbb{R}^{n - 1}.$
%\end{theorem}

%\begin{remark}{\rm
%As an immediate consequence of Tarski-Seidenberg Theorem, we get semi-algebraicity of any set $\{ x \in A \ : \ \exists y \in B,  (x, y) \in S \},$  provided that $A ,  B,$  and $S$  are semi-algebraic sets in the corresponding spaces.   It follows that also $\{ x \in A \ : \ \forall y \in B,  (x, y) \in S \}$ is
%a semi-algebraic set as its complement is the union of the complement of $A$  and the set
%$\{ x \in A \ : \ \exists y \in B,  (x, y) \not\in S \}.$ Thus, if we have a finite collection of semi-algebraic sets,
%then any set obtained from them with the help of a finite chain of quantifiers is also semi-algebraic.
%}\end{remark}

We give a version of the Curve Selection Lemma which will be used in the paper. For more details, see \cite{Milnor1968}, \cite{Nemethi1992} and see \cite{Dinh2013} for a complete proof.

\begin{lemma}[Curve Selection Lemma at infinity]\label{Lemma1}
Let $A\subset \mathbb{R}^n$ be a semi-algebraic set, and let $F := (f_1, \ldots,f_p) \colon  \mathbb{R}^n \to \mathbb{R}^p$ be a semi-algebraic map. Assume that there exists a sequence $x^k\in A$ such that $\lim_{k \to \infty} \| x^k  \| = \infty$ and $\lim_{k \to \infty} F(x^k)  = y \in(\overline{\mathbb{R}})^p,$ where $\overline{\mathbb{R}} := \mathbb{R} \cup \{\pm \infty\}.$ Then there exists a smooth semi-algebraic curve $\varphi \colon (0, \epsilon)\to \mathbb{R}^n$ such that $\varphi(t) \in A$ for all $t \in (0, \epsilon), \lim_{t \to 0} \|\varphi(t)\| = \infty,$ and $\lim_{t \to 0} F(\varphi(t)) = y.$
\end{lemma}

The following result is useful in the next section (see, e.g., \cite{Dries1996}, \cite{Miller1994}).

\begin{lemma}[Growth Dichotomy Lemma] \label{GrowthDichotomyLemma}
Let $f \colon (0, \epsilon) \rightarrow {\Bbb R}$ be a semi-algebraic function with $f(t) \ne 0$ for all $t \in (0, \epsilon).$ Then
there exist constants $c \ne 0$ and $q \in {\Bbb Q}$ such that $f(t) = ct^q + o(t^q)$ as $t \to 0^+.$
\end{lemma}

\subsection{Newton polyhedra}

In many problems, the combinatorial informations of polynomial maps are important and can be found in their Newton polyhedra. In this subsection, we recall the definition of Newton polyhedra following Kouchnirenko and Khovanskii (see \cite{Kouchnirenko1976}, \cite{{Khovanskii1978}}).

Let us begin with some notations which will be used throughout this work. We consider a fixed coordinate system $x_1, \ldots, x_n \in {\Bbb R}^n.$ Let $J \subset \{1, \ldots, n\},$ then we define
$${\Bbb R}^J := \{x \in {\Bbb R}^n \ : \ x_j = 0, \textrm{ for all } j \not \in J\}.$$

We denote by ${\Bbb R}_{\ge 0}$ the set of non-negative real numbers. We also set ${\Bbb Z}_{\ge 0} := {\Bbb R}_{\ge 0} \cap {\Bbb Z}.$ If $\kappa = (\kappa_1, \ldots, \kappa_n) \in {\Bbb Z}_{\ge 0}^n,$ we denote by $x^\kappa$ the monomial
$x_1^{\kappa_1} \cdots x_n^{\kappa_n}$ and by $| \kappa|$ the sum $\kappa_1 + \cdots + \kappa_n.$

\begin{definition} {\rm
A subset $\Gamma \subset {\Bbb R}^n_{\ge 0}$ is said to be a {\em Newton polyhedron at infinity,}
if there exists some finite subset $A \subset {\Bbb Z}^n_{\ge 0}$  such that $\Gamma$ is equal to the convex hull in ${\Bbb R}^n$ of $A \cup \{0\}.$ Hence we say that $\Gamma$ is the Newton polyhedron at infinity determined by $A$ and we write $\Gamma = \Gamma(A).$ We say that a Newton polyhedron at infinity $\Gamma \subset {\Bbb R}^n_{\ge 0}$ is {\em convenient} if it intersects each coordinate axis in a point different from the origin, that is, if for any $i \in \{1, \ldots, n\}$  there exists some integer $m_j > 0$ such that $m_j e_j \in \Gamma,$ where $\{e_1, \ldots, e_n\}$ denotes the canonical basis in ${\Bbb R}^n.$
}\end{definition}

Given a Newton polyhedron at infinity $\Gamma \subset {\Bbb R}^n_{\ge 0}$ and a vector $q \in {\Bbb R}^n,$ we define
\begin{eqnarray*}
d(q, \Gamma) &:=& \min \{\langle q, \kappa \rangle \ : \ \kappa \in \Gamma\}, \\
\Delta(q, \Gamma) &:=& \{\kappa \in \Gamma \ : \ \langle q, \kappa \rangle = d(q, \Gamma) \}.
\end{eqnarray*}
We say that a subset $\Delta$ of $\Gamma$ is a {\em face} of $\Gamma$ if there exists a vector $q \in {\Bbb R}^n$ such that $\Delta = \Delta(q, \Gamma).$ The dimension of a face $\Delta$ is defined as the minimum of the dimensions of the affine subspaces containing $\Delta.$ The faces of $\Gamma$ of dimension $0$ are called the {\em vertices} of $\Gamma.$ We denote by $\Gamma_\infty$ the set of the faces of $\Gamma$ which do not contain the origin $0$ in ${\Bbb R}^n.$

Let $\Gamma_1, \ldots, \Gamma_p$ be a collection of $p$ Newton polyhedra at infinity in ${\Bbb R}^n_{\ge 0},$ for some $p \ge 1.$ The {\em Minkowski
sum} of $\Gamma_1, \ldots, \Gamma_p$ is defined as the set
$$\Gamma_1+ \cdots + \Gamma_p = \{\kappa^1 + \cdots + \kappa^p \ : \ \kappa^i \in \Gamma_i, \textrm{ for all } i = 1, \ldots, p\}.$$
By definition, $\Gamma_1+ \cdots + \Gamma_p$ is again a Newton polyhedron at infinity. Moreover, by applying the definitions given above, it is easy to check that
\begin{eqnarray*}
d(q, \Gamma_1 + \cdots + \Gamma_p) &=& d(q, \Gamma_1) + \cdots + d(q, \Gamma_p), \\
\Delta(q, \Gamma_1 + \cdots + \Gamma_p) &=& \Delta(q, \Gamma_1) + \cdots + \Delta(q, \Gamma_p),
\end{eqnarray*}
for all $q \in \mathbb{R}^n.$ As an application of these relations, we obtain the following lemma whose proof can be found in \cite{Dinh2013}.

\begin{lemma} \label{Lemma3} 
{\rm (i)} Assume that $\Gamma$ is a convenient Newton polyhedron at infinity. Let $\Delta$ be a face of $\Gamma$ and let $q=(q_1, \ldots, q_n)\in{\Bbb R}^n$ such that $\Delta = \Delta(q, \Gamma).$ Then the following conditions are equivalent:
\begin{enumerate}
\item[{\rm (i1)}] $\Delta\in\Gamma_\infty$;
\item[{\rm (i2)}] $d(q,\Gamma)<0$;
\item[{\rm (i3)}] $\min_{j = 1, \ldots, n} q_j < 0$.
\end{enumerate}
{\rm (ii)} Assume that $\Gamma_1,\ldots, \Gamma_p$ are some Newton polyhedra at infinity. Let $\Delta$ be a face of the Minkowski sum $\Gamma := \Gamma_1+ \cdots + \Gamma_p.$ Then the following statements hold:
\begin{enumerate}
\item[{\rm (ii1)}] There exists a unique collection of faces $\Delta_1, \ldots, \Delta_p$ of $\Gamma_1, \ldots, \Gamma_p,$ respectively, such that $$\Delta = \Delta_1 + \cdots + \Delta_p.$$
\item[{\rm (ii2)}] If $\Gamma_1,\ldots, \Gamma_p$ are convenient, then $\Gamma_\infty \subset \Gamma_{1, \infty}+ \cdots + \Gamma_{p, \infty}.$
\end{enumerate}
\end{lemma}

Let $f \colon {\Bbb R}^n \to {\Bbb R}$ be a polynomial function. Suppose that $f$ is written as $f = \sum_{\kappa} a_\kappa x^\kappa.$ Then
the support of $f,$ denoted by $\mathrm{supp}(f),$ is defined as the set of those $\kappa  \in {\Bbb Z}_{\ge 0}^n$ such that $a_\kappa \ne 0.$ We denote the set $\Gamma(\mathrm{supp}(f))$ by $\Gamma(f).$ This set will be called the {\em Newton polyhedron at infinity} of $f.$ The polynomial $f$ is said to be convenient when $\Gamma(f)$ is convenient. If $f \equiv 0,$ then we set $\Gamma(f) = \emptyset.$ Note that, if $f$ is convenient, then for each nonempty subset $J$ of $\{1, \ldots, n\},$ we have $\Gamma(f) \cap {\Bbb R}^J = \Gamma(f|_{{\Bbb R}^J}).$ The {\em Newton boundary at infinity} of $f$, denoted by $\Gamma_{\infty}(f),$ is defined as the set of the faces of $\Gamma(f)$ which do not contain the origin $0$ in ${\Bbb R}^n.$

Let us fix a face $\Delta$ of $\Gamma_\infty(f).$  We define the {\em principal part of $f$ at infinity with respect to $\Delta,$} denoted by $f_\Delta,$ as the sum of those terms $a_\kappa x^\kappa$ such that $\kappa \in \Delta.$

\begin{remark}{\rm
By definition, for each face $\Delta$ of $\Gamma_\infty$ there exists a vector $q = (q_1, \ldots, q_n) \in {\Bbb R}^n$ with $\min_{j = 1, \ldots, n} q_j < 0$ such that $\Delta = \Delta(q, \Gamma).$
}\end{remark}

\subsection{Non-degeneracy at infinity}
In \cite{Khovanskii1978} (see also \cite{Kouchnirenko1976}), Khovanskii introduced a condition of non-degeneracy of  complex analytic maps $F \colon ({\Bbb C}^n, 0) \rightarrow ({\Bbb C}^p, 0)$ in terms of the Newton polyhedra of the component functions of $F.$ The class of non-degenerate maps is sufficiently large and plays an important role in Singularity Theory and Algebraic Geometry (see, for instance, \cite{Arnold1985}, \cite{Gaffney1992}, \cite{Oka1997}).  
We will apply this condition for real polynomial maps. First we need to introduce some notations. 

Let $F := (f_1, \ldots, f_p) \colon {\Bbb R}^n \rightarrow {\Bbb R}^p, 1 \le p \le n,$ be a polynomial map. Let $\Gamma(F)$ denote the Minkowski sum $\Gamma(f_1) + \cdots + \Gamma(f_p),$ and we denote by $\Gamma_\infty(F)$ the set of faces of $\Gamma(F)$ which do not contain the origin $0$ in ${\Bbb R}^n.$ Let $\Delta$ be a face of the $\Gamma(F).$ According to Lemma~\ref{Lemma3}, let us consider the decomposition $\Delta = \Delta_1 + \cdots + \Delta_p,$ where $\Delta_i$ is a face of $\Gamma(f_{i}),$ for all $i = 1, \ldots, p.$ We denote by $F_\Delta$ the polynomial map $(f_{1, \Delta_1}, \ldots, f_{p, \Delta_p})  \colon {\Bbb R}^n \rightarrow {\Bbb R}^p.$ 

\begin{definition}{\rm
The polynomial map $F = (f_1, \ldots, f_p)$ is called {\em convenient} if all $f_i$ are convenient, for $i = 1, \ldots, p.$ 

We say that $F$ is {\em non-degenerate at infinity} if and only if for any face $\Delta$ of $\Gamma_\infty(F)$ and for all $x \in ({\Bbb R} \setminus \{0\})^n,$ we have $\mathrm{rank} M_\Delta=p$ where
$$M_\Delta :=
\begin{pmatrix}
x_1\frac{\partial f_{1, \Delta_1}}{\partial x_1}(x) & \cdots & x_n \frac{\partial f_{1, \Delta_1}}{\partial x_n}(x) & f_{1,\Delta_1}(x) & \cdots & 0 \\
\vdots & \cdots & \vdots & & \ddots & \\
x_1\frac{\partial f_{p, \Delta_p}}{\partial x_1}(x) & \cdots & x_n \frac{\partial f_{p, \Delta_p}}{\partial x_n}(x) & 0 & \cdots & f_{p,\Delta_p}(x)
\end{pmatrix}.$$
}\end{definition}

\begin{remark}{\rm 
Compared to $F = (f_1, \ldots, f_p),$ the polynomial map $F_\Delta = (f_{1, \Delta_1}, \ldots, f_{p, \Delta_p})$ has the following two remarkable properties:
\begin{itemize}
\item The sparsity, which means that the number of monomials in $f_{i, \Delta_i}$ is much less than that in $f_i.$ By the Khovanskii's theory of fewnomials, see \cite{Khovanskii1991}, working with $F_\Delta$ is easier than with $F$.

\item $F_{\Delta}$ is quasi-homogeneous, i.e., there exists a vector $q \in \mathbb{R}^n,$ with $\min_j q_j < 0,$ such that for each $i = 1, \ldots, p,$ we have 
$\Delta_i = \Delta_i(q, \Gamma(f_i))$ and 
$$f_{i, \Delta_i} (t^{q_1}x_1, \ldots, t^{q_n}x_n) = t^{d_i} f_{i, \Delta_i} (x_1, \dots, x_n),$$ for all $t \in \mathbb{R},$ and for all $(x_1, \ldots, x_n) \in \mathbb{R}^n,$ where $d_i := d(q, \Gamma(f_i)).$
\end{itemize}
These facts, in many contexts, allows us to check easily the non-degenerate condition.
}\end{remark}

\section{The existence of a H\"{o}lder-type global error bound} \label{SectionErrorBound}

In this section we give a necessary and sufficient condition for the existence of a H\"{o}lder-type global error bound for a semi-algebraic function. This result extends Theorem A of \cite{HaHV2013} from polynomial functions to semi-algebraic functions.

Let $f \colon {\Bbb R}^n \rightarrow {\Bbb R}$ be a continuous semi-algebraic function. Assume that $S := \{x \in {\Bbb R}^n \ : \ f(x) \le 0 \} \ne \emptyset.$ Let $[f(x)]_+ := \max\{f(x), 0\}.$

\begin{theorem}\label{NSOAT}
With the notations above, the following two statements are equivalent.
\begin{enumerate}
\item[{\rm (i)}] For any sequence $x^k \in \mathbb{R}^n \setminus S, x^k \to \infty$, we have
\begin{enumerate}
\item[(i1)] if $f(x^k) \to 0$ then $d(x^k,S) \to 0;$
\item[(i2)] if $d(x^k,S) \to \infty$ then $f(x^k) \to \infty.$
\end{enumerate}

\item[{\rm (ii)}] There exist some constants $c > 0, \alpha > 0,$ and $\beta > 0$ such that
$$c d(x, S) \le [f(x)]_+^{\alpha} + [f(x)]_+^{\beta} \quad \textrm{ for all } \quad x \in \mathbb{R}^n.$$
\end{enumerate}
\end{theorem}

The proof is essentially the same as the proof of \cite[Theorem A]{HaHV2013} (see also \cite[Proposition 3.10]{Dinh2012}), in fact, the theorem follows from the next two lemmas that we leave the reader verifying the details.% but we shall include the full proof for the sake of completeness. The case $S = \R^n$ is trivial so assume that $S \ne \R^n.$ Let
%$$\mu(t) := \sup_{x\in f^{-1}(t)}d(x,S), \quad \textrm{ for } t \ge 0.$$
%Then the theorem follows from the next two lemmas.

\begin{lemma}[H\"older-type error bound ``near to $S$"]\label{near}
The following two statements are equivalent.
\begin{enumerate}
\item[{\rm(i)}] For any sequence $x^k \in \R^n \setminus S,$ with $x^k \to \infty,$ it holds that
$$f(x^k) \to 0 \quad \Longrightarrow \quad d(x^k,S) \to 0;$$
\item[{\rm(ii)}] There exist some constants $c > 0, \delta > 0,$ and $\alpha > 0$ such that
$$c d(x,S) \le [f(x)]_+^{\alpha}  \quad \textrm{ for all } \quad x \in f^{-1}((-\infty, \delta]).$$
\end{enumerate}
\end{lemma}
%\begin{proof}
%The conclusion will follow if we show that  [(i) $\Rightarrow$ (ii)] as [(ii) $\Rightarrow$ (i)] follows easily.

%Since $S \ne \R^n,$ there exists $\delta_1 > 0$ such that $f^{-1}(t) \ne \emptyset$ for all $t \in [0, \delta_1].$ Then the condition~ (i) implies that there exists $0 < \delta_2 \le \delta_1$ such that $\mu(t) < +\infty$ for all $t \in [0, \delta_2]$ and $\mu(t) \to 0$ when $t \to 0.$ In view of Tarski-Seidenberg Theorem \ref{TarskiSeidenbergTheorem}, the function $\mu$ is semi-algebraic on $[0, \delta_1].$
%By Growth Dichotomy Lemma (Lemma \ref{GrowthDichotomyLemma}), we can expand the function $\mu$ at $t = 0$ to get
%$$\mu(t) = at^\alpha + \textrm{ higher order term in } t,  \quad 0 < t \ll 1,$$
%where $a$ and $\alpha$ are some positive constants. We deduce that there exist $c > 0$ and $0 < \delta \le \delta_2$ such that
%$$c d(x,S) \le [f(x)]_+^{\alpha} \quad \textrm{ for all } \quad x \in f^{-1}((-\infty, \delta]).$$
%\end{proof}

\begin{lemma}[H\"older-type error bound ``far from $S$"]\label{far}
Suppose that for any sequence $x^k \in \R^n \setminus S,$ with $x^k \to \infty,$ it holds that
$$d(x^k,S) \to \infty  \quad \Longrightarrow \quad f(x^k) \to \infty;$$
Then there exist some constants $c > 0, r > 0,$ and $\beta > 0$ such that
$$c d(x,S)  \le [f(x)]_+^{\beta} \quad \textrm{ for all } \quad x \in f^{-1}([r, + \infty)).$$
\end{lemma}

\section{The Palais-Smale condition and H\"{o}lder-type global error bounds} \label{SectionPalaisSmale}

The relation between the Palais-Smale condition and the existence of error bounds is well-known, see, for example, 
\cite{Auslender1989}, \cite{Auslender1993}, \cite{Corvellec2008}, \cite{HaHV2013}, \cite{Ioffe2000}, \cite{Lemaire1992},  \cite{Penot1998}.
In this part, we describe this relation in a more convenient form for our present purposes.

First of all, we recall the notion of subdifferential-that is, an appropriate multivalued operator playing the role of the usual gradient map-which is crucial for our considerations. For nonsmooth analysis we refer to the comprehensive texts \cite{Clarke1983}, \cite{Clarke1998}, \cite{Mordukhovich2006}, \cite{Rockafellar1998}.

\begin{definition}{\rm
\begin{enumerate}
  \item[(i)] The {\em Fr\'echet subdifferential} $\hat{\partial} f(x)$ of a continuous function $f \colon {\Bbb R}^n \rightarrow {\Bbb R}$ at $x \in {\Bbb R}^n$ is given by
$$\hat{\partial} f(x) := \left \{ v \in {\Bbb R}^n \ : \ \liminf_{\| h \| \to 0, \ h \ne 0} \frac{f(x + h) - f(x) - \langle v, h \rangle}{\| h \|} \ge 0 \right \}.$$
  \item[(ii)]  The {\em limiting subdifferential} at $x \in {\Bbb R}^n,$ denoted by ${\partial} f(x),$ is the set of all cluster points of sequences $\{v^k\}_{k \ge 1}$ such that $v^k\in \hat{\partial} f(x^k)$ and $(x^k, f(x^k)) \to (x, f(x))$ as $k \to \infty.$
\end{enumerate}
}\end{definition}

\begin{remark}{\rm
It is a well-known result of variational analysis that
$\hat{\partial} f(x)$ (and a fortiori $\partial f(x)$) is not empty
in a dense subset of the domain of $f$ (see \cite{Rockafellar1998},
for example).
}\end{remark}

\begin{definition}{\rm
Using the limiting subdifferential $\partial f,$ we define the {\em
nonsmooth slope} of $f$ by
$${\frak m}_f(x) := \inf \{ \|v\| \ : \ v \in {\partial} f(x) \}.$$
By definition, ${\frak m}_f(x) = + \infty$ whenever ${\partial} f(x)
= \emptyset.$
}\end{definition}

\begin{remark}{\rm
(i) If the function $f$ is of class $C^1,$ the above notion coincides
with the usual concept of gradient; that is, ${\partial} f(x) =
\hat{\partial} f(x) = \{\nabla f(x) \},$ and hence ${\frak m}_f(x) =
\|\nabla f(x) \|.$

(ii) By Tarski-Seidenberg Theorem (see \cite{Benedetti1991}, \cite{Bochnak1998}), it is not hard to show that if the function $f$ is semi-algebraic then so is ${\frak m}_f.$
}\end{remark}

\begin{definition}{\rm
Given a continuous function $f \colon {\Bbb R}^n \rightarrow {\Bbb R}$ and a real number $t,$ we say that $f$ satisfies the {\em Palais-Smale condition at the level} $t,$ if every sequence $\{x^k\}_{k \in {\Bbb N}} \subset {\Bbb R}^n$ such that $f(x^k) \to t$ and ${\frak m}_f(x^k) \to 0$ as $k \to \infty$ possesses a convergence subsequence.
}\end{definition}

The following result extends \cite[Theorem B]{HaHV2013} from polynomial functions to semi-algebraic functions.

\begin{theorem} \label{TheoremPalais-Smale-Global-Holderian}
Let $f \colon {\Bbb R}^n \rightarrow {\Bbb R}$ be a continuous semi-algebraic function. Assume that $S := \{x \in {\Bbb R}^n \ : \ f(x) \le 0 \} \ne \emptyset.$ If $f$ satisfies the Palais-Smale condition at each level $t \ge 0$, then there exist some constants $c > 0, \alpha > 0,$ and $\beta > 0$ such that
$$c d(x,S) \le [f(x)]_+^{\alpha} + [f(x)]_+^{\beta} \quad \textrm{ for all } \quad x \in \mathbb{R}^n.$$
\end{theorem}
\begin{proof}
The proof is similar to that of \cite[Theorem B]{HaHV2013}. However, instead of using the Ekeland Variational Principle \cite{Ekeland1979}, we use a version of the variational principle of Borwein and Preiss  (see \cite{Borwein1987}, \cite[Theorem 4.2]{Clarke1998}).

It is sufficient to show that the condition (i) in Theorem~\ref{NSOAT} holds. We proceed by the method of contradiction.

We first assume that there exist a number $\delta > 0$ and a sequence $x^k \in \mathbb{R}^n \setminus S,$ with $x^k \to \infty,$ such that
$$f(x^k) \to 0 \quad \textrm{ and } \quad d(x^k,S) \ge \delta.$$
Let us consider the continuous semi-algebraic function
$$f_+ \colon {\Bbb R}^n \rightarrow {\Bbb R}, \quad x \mapsto \max\{f(x), 0\}.$$
Clearly, $\inf\limits_{x \in \mathbb{R}^n} f_+(x) = 0.$ Applying the Minimization Principle \cite[Theorem 4.2]{Clarke1998} to the function $f_+$ with data $\epsilon := f_+(x^k) = f(x^k) > 0$ and $\lambda := \frac{\delta}{4} > 0,$ we find points $y^k$ and $z^k$ in $\mathbb{R}^n$ such that
\begin{eqnarray*}
&& \|z^k - x^k\| < \lambda, \quad \|y^k - z^k\| < \lambda, \quad f_+(y^k) \le  f_+(x^k),
\end{eqnarray*}
and such that the function
\begin{eqnarray*}
&& x \mapsto f_+(x) + \frac{\epsilon}{\lambda^2} \| x - z^k\|^2
\end{eqnarray*}
is minimized over $\mathbb{R}^n$ at $y^k.$ We deduce from the above inequalities that
$$\|y^k - x^k\| \ \le \ \|z^k - x^k\| + \|y^k - z^k\| \ < \ 2\lambda \ = \ \frac{\delta}{2},$$
which yields that $\lim_{k \to \infty} \|y^k\| = \infty$ and
\begin{eqnarray*}
d(y^k, S)
&\ge& d(x^k, S) - d(x^k, y^k) \ >  \ d(x^k, S) - \frac{\delta}{2} \  \ge  \ \frac{\delta}{2}.
\end{eqnarray*}
Hence,
$$B(y^k, \frac{\delta}{2}) := \left \{x \in {\Bbb R}^n \ : \ \|x - y^k\| < \frac{\delta}{2} \right \} \subset {\Bbb R}^n \setminus S.$$In particular, we have $f_+(x) = f(x)$ for all $x \in B(y^k, \frac{\delta}{2}).$
Consequently, the function
\begin{eqnarray*}
&& x \mapsto f(x) + \frac{\epsilon}{\lambda^2} \| x - z^k\|^2
\end{eqnarray*}
attains its minimum on the open ball $B(y^k, \frac{\delta}{2})$ at $y^k.$ Then, by the Fermat's rule generalized \cite[Theorem 10.1]{Rockafellar1998}, we get
$$-2 \frac{\epsilon}{\lambda^2} (y^k - z^k) \in \partial f(y^k) .$$
Therefore
$${\frak m}_f(y^k) \ \le  \ 2 \frac{\epsilon}{\lambda^2}\| y^k - z^k \|  \ \le  \ 2 \frac{\epsilon}{\lambda} \ \le \ \frac{8f(x^k)}{\delta}.$$
By letting $k$ tend to infinity, we obtain
$$\lim_{k \to \infty} \|y^k\| = \infty, \quad \lim_{k \to \infty} f(y^k) = 0, \quad \textrm{ and } \quad \lim_{k \to \infty} {\frak m}_f(y^k) = 0.$$So, $f$ does not satisfy the Palais-Smale condition at the value $t = 0,$ and a contradiction follows.

We next suppose that there exist a number $M > 0$ and a sequence $x^k \in \mathbb{R}^n \setminus S,$ with $x^k \to \infty,$ such that
$$d(x^k,S) \to \infty \quad \textrm{ and } \quad 0 < f(x^k) \le M.$$
Again, we see that $\inf_{x \in {\Bbb R}^n} f_+(x) = 0.$   We now apply the Minimization Principle \cite[Theorem 4.2]{Clarke1998} to the function $f_+$ with data $\epsilon := f_+(x^k) = f(x^k) > 0$ and $\lambda := \frac{d(x^k,S)}{4} > 0;$ there exist points $y^k$ and $z^k$ in ${\Bbb R}^n$ with
\begin{eqnarray*}
&& \|z^k - x^k\| < \lambda, \quad \|y^k - z^k\| < \lambda, \quad f_+(y^k) \le  f_+(x^k),
\end{eqnarray*}
and having the property that the function
\begin{eqnarray*}
&& x \mapsto f_+(x) + \frac{\epsilon}{\lambda^2} \| x - z^k\|^2
\end{eqnarray*}
has a unique minimum at  $y^k.$ We deduce from the above inequalities that
\begin{eqnarray*}
d(y^k, S)
&\ge& d(x^k, S) - d(x^k, y^k) \\
&\ge& d(x^k, S) - 2\lambda \ = \ \frac{d(x^k,S)}{2},
\end{eqnarray*}
which yields $\lim_{k \to \infty} d(y^k, S)  = \infty.$ In particular, we get $y^k \in \mathbb{R}^n \setminus S$ and $y^k \to \infty$.

By an argument as above, we can easily deduce again that
$${\frak m}_f(y^k)  \ \le \ 2 \frac{\epsilon}{\lambda^2}\| y^k - z^k \|  \ \le \  2 \frac{\epsilon}{\lambda} \ \le \ \frac{8M}{d(x^k,S)}.$$
Hence,
$$\lim_{k \to \infty} {\frak m}_f(y^k) = 0.$$
Note that $0 <  f(y^k) \le  f(x^k) \le M$ for all $k \ge 1.$ Hence, by passing to subsequences if necessary, we may assume that there exists the limit $t := \lim_{k \to \infty}  f(y^k).$ Therefore $f$ does not satisfy  the Palais-Smale condition at $t,$ which is a contradiction. The proof of Theorem~ \ref{TheoremPalais-Smale-Global-Holderian} is complete.
\end{proof}

\section{Goodness at infinity and H\"{o}lder-type global error bounds} \label{SectionGoodatinfinity}
The aim of this section is to establish a H\"{o}lder-type global error bound (with the exponent $\beta = 1$) for continuous semi-algebraic functions which have good asymptotic behavior at infinity. Let $f \colon {\Bbb R}^n \rightarrow {\Bbb R}$ be a continuous function. For $x\in{\Bbb R}^n$, set 
$$f_+(x) := \max\{f(x), 0\}.$$
Then $f_+$ is also a continuous function. Let us begin with the following definition.

\begin{definition}{\rm
A continuous function $f$ is said to be {\em good at infinity} if there exist some constants $c > 0$ and $R > 0$ such that
\begin{equation*}
{\frak m}_{f} (x) \ge c \quad \textrm{ for all } \quad x \in f^{-1}((0, +\infty)) \ \textrm{ and } \ \|x\| \ge R.
\end{equation*}
}\end{definition}

The main result of this section is as follows:

\begin{theorem} \label{GoodnessTheorem}
Let $f$ be a continuous semi-algebraic function which is good at infinity. Assume that $S := \{x \in {\Bbb R}^n \ : \ f(x) \le 0 \} \ne \emptyset.$ Then there exist some constants $c > 0$ and $\alpha > 0$ such that
$$c d(x,S) \le [f(x)]_+^{\alpha} + [f(x)]_+  \quad \textrm{ for all } \quad x \in \mathbb{R}^n.$$
\end{theorem}

\begin{proof} 
Let us consider the continuous semi-algebraic function
$$f_+ \colon {\Bbb R}^n \rightarrow {\Bbb R}, \quad x \mapsto \max\{f(x), 0\}.$$ 
By definition, if $f(x) > 0$ then $f_+(x) = f(x),$ $\partial f_+(x) = \partial f(x)$ and ${\frak m}_{f_+}(x) = {\frak m}_{f}(x).$ 
Since $f$ is good at infinity, there exist some constants $c_1 > 0$ and $R > 0$ such that 
\begin{equation}\label{Good1}
{\frak m}_{f_+} (x)={\frak m}_{f} (x) \ge c_1 \quad \textrm{ for all } \quad x \in f^{-1}((0, +\infty))=f_+^{-1}((0, +\infty)) \ \textrm{ and } \ \|x\| \ge R.
\end{equation}

Thanks to the classical \L ojasiewicz inequality \cite{Hormander1958}, \cite{Lojasiewicz1958}, there are constants $c_2 > 0$ and $\alpha > 0$ such that
\begin{equation}\label{Good2}
c_2 d(x, S)  \le f_+(x) ^{\alpha} \quad \textrm{ for all } \quad \|x\| \le R.
\end{equation}

Let $x \in {\Bbb R}^n$ be such that $x \in {f_+}^{-1}((0, +\infty))$ and $\|x \| > R.$ By \cite[Corollary 4.1]{Bolte2007}, there exists a maximal absolutely continuous curve $u \colon [0, \infty) \to {\Bbb R}^n$ of the dynamical system 
$$0 \in \dot{u}(s) + \partial [f_+(u(s))]$$
satisfying $u(0)=x.$ In addition, the function $s \mapsto (f_+ \circ u)(s)$ is absolutely continuous and strictly decreasing on $[0, + \infty).$ By \cite[Corollary 4.2]{Bolte2007}, we have for almost all $s \in [0, +\infty),$ 
\begin{equation}\label{Good3}
\| \dot{u}(s)\|=\frak m_{f_+}(u(s)) \quad \textrm{ and } \quad \frac{d}{ds}({f_+} \circ u)(s) = -[\frak m_{f_+}(u(s))]^2.
\end{equation}

We have the following remark.  Suppose that $f_+(u(s)) > 0$ and $\| u(s)\| \ge R$ for all $s \in [t_1, t_2],$ for some  $0 \le t_1 < t_2.$ It follows from the relations (\ref{Good1}) and (\ref{Good3}) that
\begin{eqnarray*}
f_+(u(t_1) ) - f_+(u(t_2)) &=& - \int_{t_1}^{t_2} \frac{d}{ds} (f_+ \circ u)(s) ds \ = \ \int_{t_1}^{t_2}  [{\frak m}_{f_+}(u(s))]^2ds \\
&\ge&   \int_{t_1}^{t_2}  c_1  {\frak m}_{f_+} (u(s)) ds  = \int_{t_1}^{t_2} c_1 \|\dot{u}(s)\|ds, 
\end{eqnarray*}
which yields
\begin{eqnarray}\label{Good4}
f_+(u(t_1) ) - f_+(u(t_2)) &\ge& c_1 \| u(t_1) - u(t_2)\|.
\end{eqnarray}
Hence the curve $u$ has finite length, and so it is bounded. In view of \cite[Theomrem 4.5]{Bolte2007}, there exists the limit
$a := \lim_{s \to \infty} f(u(s)).$ In addition, we have ${\frak m}_f(a) = 0.$ Let 
$$t := \inf\{ s \ : \ \|u(s) \| > R \}.$$
There are two cases to be considered.

\subsubsection*{Case 1: $t = \infty;$ i.e.,  $\|u(s)\| > R$ for all $s \ge 0$} \

Since 
${\frak m}_f(a) = 0,$ it follows from the inequality (\ref{Good1}) that $f_+(a) = 0.$ Therefore, by (\ref{Good4}), we obtain
\begin{eqnarray*}
f_+(x)  \ = \ f_+(x) - f_+(a) & \ge & c_1 \|u(0) - a\| \ = \ 
c_1 \|x - a\| \ \ge \ c_1 d(x, S).
\end{eqnarray*}

\subsubsection*{Case 2: $t  < \infty$} \

We have $\| u(t) \| = R.$ Then it follows from (\ref{Good1}), (\ref{Good2}) and (\ref{Good4}) that
\begin{eqnarray*}
d(x, S) &\le& d(x,u(t))+d(u(t), S)\\
&\le& \frac{f_+(x) - f_+(u(t))}{c_1} + \frac{(f_+(u(t)))^{\alpha}}{c_2} \\
&\le& \frac{f_+(x)}{c_1} + \frac{(f_+(x))^{\alpha}}{c_2}.
\end{eqnarray*}

In summary, in both cases, we have
$$c d(x, S) \le f_+(x) + (f_+(x))^{\alpha},$$
where $c := \min\{c_1, c_2\} > 0.$ This, together with (\ref{Good2}), implies the required result.
\end{proof}

We deduce immediately the following corollary from the proof of Theorem \ref{GoodnessTheorem} (see also \cite[Theorem 3.1]{Wu2001}).

\begin{corollary}  \label{CorollaryGoodness}
Let $f$ be a continuous semi-algebraic function. Suppose that $S := \{x \in {\Bbb R}^n \ : \ f(x) \le 0 \} \ne \emptyset$ and there exists a constant $c > 0$ such that
$${\frak m}_{f} (x) \ge c, \quad \textrm{ for all } \quad x \in f^{-1}((0, + \infty)).$$
Then the following linear global error bound holds
$$c d(x,S) \le [f(x)]_+, \quad \textrm{ for all } \quad x \in \mathbb{R}^n.$$
\end{corollary}

Let $A \in \mathbb{R}^{n \times n}$ be a nonzero symmetric matrix. It is well-known that $A$ has $n$ real eigenvalues $\lambda_i, i = 1, \ldots, n.$ Then we make use of the following notation
$$\lambda(A) := \min \{ | \lambda_i| \ : \ \lambda_i \ne 0, i = 1, \ldots, n \}.$$
The following result gives a H\"{o}lder-type global error bound result for the zero set of a single quadratic function.

\begin{corollary}
Let $f(x) := \displaystyle\frac{1}{2} x^T Ax + x^T b + c$ be a quadratic function in $\mathbb{R}^n,$ where $A \in \mathbb{R}^{n \times n}$ is a nonzero symmetric matrix, $b$ is a vector in $\mathbb{R}^n$ and $c$ is a real number. Let $\bar x \in \mathbb{R}^n$ be such that $\nabla f(\bar x) = A \bar x + b = 0.$
Then we have for all $x \in \mathbb{R}^n,$
\begin{eqnarray*}
\frac{\sqrt{2 \lambda(A)}}{2} \ d(x, f^{-1}(f(\bar x))) & \le & | f(x) - f(\bar x)|^{\frac{1}{2}}.
\end{eqnarray*}
\end{corollary}
\begin{proof}
We first show the following {\em gradient inequality} (see \cite[Property 6]{Forti2006}):
\begin{eqnarray*}
\sqrt{2 \lambda(A)}  \ | f(x) - f(\bar x)|^{\frac{1}{2}} 
& \le & \|\nabla f(x) \|, \quad \textrm{ for all } x \in \mathbb{R}^n.
\end{eqnarray*}
Indeed, by the assumption, we have $b = - A \bar x.$ This implies that $\nabla f(x) = Ax + b = A (x - \bar x) $ and
\begin{eqnarray*}
f(x) - f(\bar x) 
&=& \frac{1}{2} (x^T A x - \bar x^T A \bar x) + (x - \bar x)^Tb \\
&=& \frac{1}{2} (x^T A x - \bar x^T A \bar x) - (x - \bar x)^TA \bar x \\
&=& \frac{1}{2} (x - \bar x)^T A (x - \bar x).
\end{eqnarray*}

On the other hand, since the matrix $A$ is symmetric, $A$ has $n$ real eigenvalues $\lambda_i, i = 1, \ldots, n.$ Then we can write
$$A = R \ \mathrm{diag}(\lambda_1, \ldots, \lambda_\kappa, 0, \ldots, 0) R^T,$$
where $R  \in \mathbb{R}^{n \times n}$ is a suitable orthonormal matrix, and $\lambda_i, i = 1, \ldots, \kappa,$ is the nonzero eigenvalues of $A.$

Now let $x\in\mathbb{R}^{n}$ and set $z := R^T(x - \bar x) \in \mathbb{R}^{n}.$ We have
$$\| \nabla f(x) \|^2 = \| A (x - \bar x)\|^2 = \sum_{i = 1}^\kappa \lambda_i^2 z_i^2$$ 
and
\begin{eqnarray*}
| f(x) - f(\bar x) |
&=& \frac{1}{2} \left | (x - \bar x)^T A (x - \bar x) \right | 
\ = \ \frac{1}{2}  \left | \sum_{i = 1}^\kappa \lambda_i z_i^2 \right| 
\ \le \ \frac{1}{2}  \sum_{i = 1}^\kappa | \lambda_i | z_i^2.
\end{eqnarray*}
which implies the gradient inequality.

We now consider the continuous semi-algebraic function $g \colon \mathbb{R}^n \rightarrow \mathbb{R}$ defined by 
$g(x) := |f(x) - f(\bar x)|^{\frac{1}{2}}.$ Then it easily follows from the above gradient inequality that 
${\frak m}_{g}(x) \ge \frac{\sqrt{2 \lambda(A)}}{2}$ for any $x$ with
$g(x) > 0.$
This, together with Corollary \ref{CorollaryGoodness}, implies the desired result.
\end{proof}

\begin{remark}{\rm
Recall that the authors of \cite{Luo2000-1} (see also \cite{Dinh2012}) established the following H\"{o}lder-type global error bound result for the zero set of a single quadratic function $f$: There exists some constant $c > 0$ such that
$$c d(x, f^{-1}(0)) \le |f(x)|^{\frac{1}{2}} + |f(x)|, \quad \textrm{ for all }  \quad x \in \mathbb{R}^n.$$
However,  neither result gives  any  clue  for  computing  the  constant $c$  in general.
}\end{remark}

%\section{Non-degeneracy at infinity and H\"{o}lder-type global error bounds} \label{SectionHolder}
\section{Proof of the main result} \label{SectionHolder}

%In this part, we establish the H\"{o}lder-type global error bound for polynomial maps which are non-degenerate at infinity, where the corresponding exponents $\alpha$ and $\beta$ are explicitly determined. %To begin with, for any positive integers $d, n,$ and $p,$ we let
%$$\mathcal{H}(d, n, p) := d(6d - 3)^{n + p - 1}.$$

%\begin{theorem} {\rm (Compare \cite[Theorem C]{HaHV2013})} \label{HolderTypeTheorem}
%Let $F := (f_1, \ldots, f_p) \colon {\Bbb R}^n \rightarrow {\Bbb R}^p, 1 \le p \le n,$ be a polynomial map. Suppose that $F$ is convenient and non-degenerate at infinity. Let $f(x) := \max_{i = 1, \ldots, p} f_i(x)$ and $S := \{x \in {\Bbb
%R}^n \ | \ f(x) \le 0\} \ne \emptyset.$ Then there exists a constant $c>0$ such that
%$$c d(x,S) \le [f(x)]_+^{\frac{1}{\mathcal{H}(d, n, p)}} + [f(x)]_+ \quad \textrm{ for all } \quad x \in \mathbb{R}^n,$$
%where $d:=\max_{i = 1, \ldots, p}\deg f_i.$
%\end{theorem}

The following lemmas are crucially used in the proof of Theorem \ref{HolderTypeTheorem}.

\begin{lemma} \label{Lemma7}
Let $F = (f_1, \ldots, f_p) \colon {\Bbb R}^n \rightarrow {\Bbb R}^p$ be a map of class $C^1$ and let $f(x) := \max_{i = 1, \ldots, p} f_i(x).$ Then $f$ is a continuous function and
\begin{eqnarray*}
{\frak m}_f(x) &=& \min_{
\begin{matrix}
\lambda_i \ge  0,\ \sum_{i \in I} \lambda_i  = 1
\end{matrix}}   \left \| \sum_{i \in I} \lambda_i \nabla f_i(x) \right \|,
\end{eqnarray*}
where $I=I(x):= \{i \ : \ f_i(x) = f(x)\}.$
\end{lemma}
\begin{proof}
The statement is a consequence of \cite[Theorem~3.46(ii)]{Mordukhovich2006} (cf. also \cite[Exercise~8.31]{Rockafellar1998}).
\end{proof}

\begin{lemma} \label{Lemma8}
Under the assumptions of Theorem \ref{HolderTypeTheorem}, there exist some constants $c > 0$ and $R > 0$ such that
$${\frak m}_f(x) \ge c \quad \textrm{ for all } \quad \|x\| \ge R.$$
In particular, the function $f$ is good at infinity.
\end{lemma}
\begin{proof}
Suppose that by contradiction there exists a sequence $\{x^k\}_{k \in {\Bbb N}} \subset {\Bbb R}^n$ such that
$$\lim_{k \to \infty} \|x^k \|= \infty \quad \quad \textrm{ and } \quad \lim_{k \to \infty} {\frak m}_f(x^k) = 0.$$
By definition, there exists a sequence $\lambda^k:=(\lambda_1^k,\ldots,\lambda_p^k)$ with $\lambda_i^k \ge  0,\ \sum_{i \in I(x^k)} \lambda_i^k=1$ such that 
$${\frak m}_f(x^k) =\left \| \sum_{i \in I(x^k)} \lambda_i^k \nabla f_i(x^k) \right \|.$$
Since the number of subsets of $\{1,\ldots,p\}$ is finite, by taking subsequences if necessary, we may assume that the set $I(x^k)$ is stable, i.e., there exists $\tilde I\subseteq \{1,\ldots,p\}$ such that $\tilde I=I(x^k)$ for all $k$.
We remark that the function ${\frak m}_f(x)$ is semi-algebraic. 
By Lemma \ref{Lemma7} and by applying Curve Selection Lemma at infinity (Lemma \ref{Lemma1}) with the following setup: the set
\begin{eqnarray*}
A:=\{(x,\lambda)\in\Bbb R^n\times\Bbb R^p:&&\lambda_i \ge  0, \ \sum_{i \in \tilde I} \lambda_i=1, \\
										  &&f_i(x)=f(x) \text { for } i\in\tilde I,\\
										  &&f_i(x)<f(x) \text { for } i\not\in\tilde I,\\
									 	  &&{\frak m}_f(x) =||\sum_{i \in \tilde I} \lambda_i \nabla f_i(x)||\}
\end{eqnarray*}
which is a semi-algebraic set, the sequence $(x^k,\lambda^k)\in A$ which tends to infinity as $k\to\infty,$ and the semi-algebraic function $x\mapsto {\frak m}_f(x),$
it follows that there exist a smooth semi-algebraic curve $\varphi(t) := (\varphi_1(t), \ldots, \varphi_n(t))$ and some smooth semi-algebraic functions $\lambda_i(t), i \in \tilde{I},$ for $0 < t \ll 1,$ such that
\begin{enumerate}
  \item [(a)] $\lim_{t \to 0} \| \varphi(t)\| = \infty;$
  \item [(b)] $f_i(\varphi(t)) = f(\varphi(t))$ for $i \in \tilde{I},$ and $f_i(\varphi(t)) < f(\varphi(t))$ for $i \not \in \tilde{I};$
  \item [(c)] $\lambda_i(t) \ge 0$ for all $i \in \tilde{I},$ and $\sum_{i \in \tilde{I}}\lambda_i(t) = 1.$
  \item [(d)] ${\frak m}_f(\varphi(t)) = \|\sum_{i \in \tilde{I}} \lambda_i(t) \nabla f_i(\varphi(t))\| \to 0$ as $t \to 0.$
\end{enumerate}

Let $J := \{j  : \ \varphi_j \not\equiv 0\}.$ By Condition (a),
$J \ne \emptyset.$ In view of Growth Dichotomy Lemma (Lemma \ref{GrowthDichotomyLemma}), for $j \in J,$ we can expand the coordinate $\varphi_j$ in terms of the parameter: say
$$\varphi_j(t) =  x_j^0 t^{q_j} + \textrm{ higher order terms in } t,$$
where $x_j^0 \ne 0.$ From Condition (a), we get $q_{j_*} := \min_{j \in J} q_j  < 0$ for some $j_* \in J.$ Note that $\|\varphi(t)\| = c t^{q_{j_*}} + o(t^{q_{j_*}})$ as $t \to 0,$ for some $c > 0.$

Since $f_i$ is convenient, $\Gamma(f_i) \cap {\Bbb R}^J \ne \emptyset.$  Let $d_i$ be the minimal value of the linear function $\sum_{j \in J}
q_j \kappa_j$ on $\Gamma(f_i) \cap {\Bbb R}^J,$ and let $\Delta_i$ be the (unique)
maximal face of $\Gamma(f_i) \cap {\Bbb R}^J$ where the linear function takes this
value. Since $f_i$ is convenient, $d_i < 0$ and $\Delta_i$ is a face of $\Gamma_\infty(f_i).$
Note that $f_{i, \Delta_i}$ does not dependent on $x_j$ for all $j \not \in J.$
By a direct calculation, then
$$f_i(\varphi(t)) = f_{i, \Delta_i}(x^0)t^{d_i} + \textrm{ higher order terms in } t,$$
where $x^0 := (x_1^0, \ldots, x_n^0)$ with $x_j^0 = 1$ for $j \not \in J.$

Let $I := \{i \in \tilde{I} \ : \ \lambda_i \not \equiv 0 \}.$ It follows from
Condition (c) that $I \ne \emptyset.$ For $i \in I,$ expand the coordinate $\lambda_i$ in terms of the parameter: say
$$\lambda_i(t) =  \lambda_i^0 t^{\theta_i} + \textrm{ higher order terms in } t,$$
where $\lambda_i^0 \ne 0.$

For $i \in I$ and $j \in J$ we have
\begin{eqnarray*}
\frac{\partial f_i}{\partial x_j}(\varphi(t))
&=& \frac{\partial f_{i, \Delta_i}}{\partial x_j}(x^0)t^{d_i - q_j}  + \textrm{ higher order terms in } t.
\end{eqnarray*}
It implies that
\begin{eqnarray*}
\sum_{i \in I} \lambda_i(t) \frac{\partial f_i}{\partial x_j}(\varphi(t))
&=& \sum_{i \in I} \left( \lambda_i^0  \frac{\partial f_{i, \Delta_i}}{\partial x_j}(x^0)t^{d_i + \theta_i - q_j}  + \textrm{ higher order terms in } t \right) \\
&=& \left( \sum_{i \in I'} \lambda_i^0  \frac{\partial f_{i, \Delta_i}}{\partial x_j}(x^0) \right) t^{\ell  - q_j}  + \textrm{ higher order terms in } t,
\end{eqnarray*}
where $\ell  := \min_{i \in I} (d_i + \theta_i)$ and $I' := \{i \in I \ : \ d_i + \theta_i = \ell \} \ne \emptyset.$

There are two cases to be considered.

\subsubsection*{Case 1: $\ell  \le q_{j_*}  := \min_{j \in J} q_j$} \

We deduce from Condition (d) that
$$\sum_{i \in I'} \lambda_i^0  \frac{\partial f_{i, \Delta_i}}{\partial x_j}(x^0) = 0, \quad \textrm{for all } \quad j \in J,$$
which yields
$$\sum_{i \in I'} \lambda_i^0  \frac{\partial f_{i, \Delta_i}}{\partial x_j}(x^0) = 0, \quad \textrm{for all } \quad j = 1, 2, \dots, n,$$
because $f_{i, \Delta_i}$ does not dependent on $x_j$ for all $j \not \in J.$ It implies easily that
\begin{equation*}
\textrm{rank}
\begin{pmatrix}
x_j^0\frac{\partial f_{i, \Delta_i}}{\partial x_j}(x^0)
\end{pmatrix}_{i \in I', 1 \le j \le n} < \# I'.
\end{equation*}
Since the map $F = (f_1, \ldots, f_p)$ is non-degenerate at infinity, there exists an index $i_0 \in I'$ such that $f_{i_0, \Delta_{i_0}}(x^0) \ne 0.$ Then, by Condition (b), we have for all $i \in \tilde{I},$
$$f(\varphi(t)) = f_i(\varphi(t)) = f_{i_0}(\varphi(t)) = f_{i_0, \Delta_{i_0}}  (x^0) t^{d_{i_0}} + \textrm{ higher order terms in } t.$$
By taking the derivative in $t$ of the function $(f \circ \varphi)(t)$, we deduce that
\begin{eqnarray*}
\frac{d (f \circ \varphi)(t)}{dt} &=& \frac{d (f_i \circ \varphi)(t)}{dt} = \left \langle \nabla f_i(\varphi(t)), \frac{d \varphi(t)}{dt} \right \rangle,
\quad \textrm{for all } \quad i \in \tilde{I}.
\end{eqnarray*}
By Condition (c), then
\begin{eqnarray*}
\frac{d (f \circ \varphi)(t)}{dt}  \ = \ \sum_{i \in \tilde{I}} \lambda_i(t) \frac{d (f \circ \varphi)(t)}{dt} &=& \left \langle \sum_{i \in \tilde{I}} \lambda_i(t) \nabla f_i(\varphi(t)), \frac{d \varphi(t)}{dt} \right \rangle.
\end{eqnarray*}
Thus
\begin{eqnarray*}
\left |  \frac{d (f \circ \varphi)(t)}{dt} \right | 
& \le &  \Big\| \sum_{i \in \tilde{I}} \lambda_i(t) \nabla f_i(\varphi(t)) \Big \| \left \| \frac{d \varphi(t)}{dt} \right \| \  = \  {\frak m}_f(\varphi(t)) \left \| \frac{d \varphi(t)}{dt} \right \|,
\end{eqnarray*}
which implies that
$${\frak m}_f(\varphi(t)) \ge c't^{d_{i_0} - q_{j_*}}  + \textrm{ higher order terms in } t,$$
for some $c' > 0.$ But this inequality contradicts Condition (d) since we know that
$$d_{i_0} \le d_{i_0} + \theta_{i_0} = \ell  \le q_{j_*}.$$

\subsubsection*{Case 2: $\ell  > q_{j_*}  := \min_{j \in J} q_j$} \

It follows from Condition (c) that  $\theta_i \ge 0$ for all $i \in I$ and $\theta_i = 0$ for some $i \in I.$ Without lost of generality, we may assume that $1 \in I$ and $\theta_1 =  0.$ Since $f_1$ is convenient, for any $j = 1, \ldots, n,$ there exists a natural number $m_j \ge 1$ such that $m_j e_j \in \Gamma_\infty(f_1).$ Then it is clear that
$$q_j m_j \ge d_1, \quad \textrm{ for all } j \in J.$$
On the other hand, we have
$$d_1 = d_1 + \theta_1 \ge \min_{i \in I} (d_i + \theta_i) = \ell.$$
Therefore
$$q_{j_*} m_{j_*} \ge d_1 \ge \ell  > q_{j_*}.$$
Since $q_{j_*}  = \min_{j \in J} q_j   < 0,$ it implies that $m_{j_*} < 1$, which is a contradiction.
\end{proof}

\begin{corollary} Under the assumptions of Theorem \ref{HolderTypeTheorem}, there exist some positive constants $c,\delta$ and $\alpha$ such that the following H\"older-type error bound ``near to $S:=\{x\in\Bbb{R}^n:{\frak m}_f(x)=0\}$" holds
$$cd(x,S)\le [{\frak m}_f(x)]^\alpha \quad \textrm{ for all } \quad x \in \ {\frak m}_f\le \delta\}.$$
\end{corollary}
\begin{proof} By Lemma \ref{Lemma8}, there exist some constants $c_1 > 0$ and $R > 0$ such that
$${\frak m}_f(x) \ge c_1 \quad \textrm{ for all } \quad \|x\| \ge R.$$
Hence there is no sequence $x^k\to\infty$ such that ${\frak m}_f(x)\to 0$. By Lemma \ref{near}, there exist some constants $c > 0, \delta > 0,$ and $\alpha > 0$ such that
$$cd(x,S)\le [{\frak m}_f(x)]^\alpha \quad \textrm{ for all } \quad x \in \ {\frak m}_f\le \delta.$$
\end{proof}

\begin{remark}{\rm
Lemma \ref{Lemma8} was proved by another method by H\`a \cite{HaHV2013} (see also \cite[Proposition 3.4]{Broughton1988}) for a single polynomial function; i.e., for the case where $p = 1.$ 
}\end{remark}

Before proving Theorem \ref{HolderTypeTheorem} which establishes that a H\"{o}lder-type global error bound holds
with an explicit exponent, we recall an error bound result on a bounded region.

\begin{lemma}\label{alpha}
Let $S$ denote the set of $x$ in ${\Bbb R}^n$ satisfying $f_1(x) \le 0, \ldots, f_p(x) \le 0,$ where each $f_i$ is a real polynomial.
Let $R$ be a positive number such that $S$ contains an element $x$ with $\|x\| \le R.$ Then, there exists a constant $c > 0$ such that
\begin{equation} \label{HolderExponent}
c  d(x, S)  \le [f(x)]_+^{\frac{2}{\mathcal{H}(2d, n, p)}} \quad \textrm{ for all } x \textrm{ with } \|x\| \le R.
\end{equation}
Here $f(x) := \max_{i = 1, \ldots, p} f_i(x)$ and $d := \max_{i = 1, \ldots, p} \deg f_i.$
\end{lemma}

\begin{proof} 
See \cite[Corollary 3.8]{LMP2013}.
\end{proof}

Now, we are in position to finish the proof of Theorem \ref{HolderTypeTheorem}.

\begin{proof} [Proof of Theorem \ref{HolderTypeTheorem}]
By Lemma \ref{Lemma8}, the continuous semi-algebraic function $f \colon \mathbb{R}^n \rightarrow \mathbb{R}, x \mapsto f(x) := \max_{i = 1, \ldots, p} f_i(x),$ is good at infinity. Then
the proof follows on the same lines as that of Theorem \ref{GoodnessTheorem}, by using 
the inequality (\ref{HolderExponent}) in Lemma \ref{alpha} instead of the inequality (\ref{Good2}). We omit the details.
\end{proof}

\begin{remark}{\rm
Theorem \ref{HolderTypeTheorem} was obtained by H\`a \cite{HaHV2013} for a single polynomial function; however, there was no explicit formula given for computing the exponents $\alpha$ and $\beta.$  Again, the present proof is different from that in the above cited paper.

%(ii) An advantage of the H\"older-type global error bound in Theorem \ref{HolderTypeTheorem} is that if we know the value of $f$ at a point $x \in \mathbb{R}^n,$ then we can say the following: 
%\begin{itemize}
%\item If $[f(x)]_+$ is small, then we know how much $x$ is close to the set $S$;
%\item If $[f(x)]_+$ is not too large, then we know how much $x$ is not too far from the set $S.$
%\end{itemize}
}\end{remark}

\section{Examples}\label{Examples}

In this section, we give some examples which illustrate Theorem \ref{HolderTypeTheorem}.

Denote the convex hull of a set of points $0, a_1, \ldots ,a_m\in\mathbb{R}^n$ by $\Gamma\{a_1,\ldots ,a_m\}$.
\begin{example}{\rm
 Consider the following polynomial map
$$F=(f_1,f_2) \colon \mathbb{R}^2\to\mathbb{R}^2, \quad (x,y)\mapsto (x+y,x^2+y^2-1).$$
Note that $f_1,f_2$ are convenient and $S$ is the half-disk $\{x+y\le 0,x^2+y^2\le 1\}$. The Newton polyhedra at infinity of $f_1$ and $f_2$, are the triangles $\Gamma(f_1)=\Gamma\{(1,0),(0,1)\}$ and $\Gamma(f_2)=\Gamma\{(2,0),(0,2),(0,0)\}$, respectively. The Minkowski sum $$\Gamma(F)=\Gamma(f_1)+\Gamma(f_2)=\Gamma\{(3,0),(0,3),(0,0)\}$$ is again a triangle. Then 
$\Gamma_\infty(F)$ has three faces which are $\Delta^1:=\{(3,0)\}=\{(1,0)\}+\{(2,0)\},\Delta^2:=\{(0,3)\}=\{(0,1)\}+\{(0,2)\},$ and $\Delta^3:=\Gamma\{(3,0),(0,3)\}=\Gamma\{(1,0),(0,1)\}+\Gamma\{(2,0),(0,2)\}$. So we have $f_{\Delta^1}=(x,x^2)$, $f_{\Delta^2}=(y,y^2),$ and $f_{\Delta^3}=(x+y,x^2+y^2)$. It is clear that the following corresponding matrices 
$$M_{\Delta^1} :=
\begin{pmatrix}
x     & 0 & x & 0 \\
2 x^2 & 0 & 0 & x^2 \\
\end{pmatrix}, \
M_{\Delta^2} :=
\begin{pmatrix}
0 & y     & y & 0 \\
0 & 2 y^2 & 0 & y^2 \\
\end{pmatrix}, \
M_{\Delta^3} :=
\begin{pmatrix}
x     & y     & x+y & 0 \\
2 x^2 & 2 y^2 & 0   & x^2+y^2 \\
\end{pmatrix}
$$
have rank $2$ on $(\mathbb{R}\setminus\{0\})^2$. Hence $F$ is non-degenerate at infinity. The reader may easily check the following H\"{o}lder-type global error bound
$$c d(x,S) \le [f(x)]_+^{\frac{1}{2}} + [f(x)]_+ \quad \textrm{ for all } \quad x \in \mathbb{R}^n,$$
and for some $c>0$. The exponent $\alpha=\frac{1}{2}$ here can be also obtained form \cite[Corollary 16.14]{Luo2000-1} by restricting on the case of systems of one linear and one convex quadratic inequality.
}\end{example}

In general, it is not easy to verify directly whenever a system of polynomials has H\"{o}lder-type global error bounds or not. However, it can be done by checking the condition of non-degeneracy at infinity.

\begin{example}{\rm
Let
$$F=(f_1,f_2) \colon \mathbb{R}^3\to\mathbb{R}^2, \quad (x,y,z)\mapsto (x^2+y^2+z^2,x+y+z^3).$$
We have
\begin{eqnarray*}
\Gamma(f_1) &=& \Gamma\{(2,0,0),(0,2,0),(0,0,2)\}, \\
\Gamma(f_2) &=& \Gamma\{(1,0,0),(0,1,0),(0,0,3)\}.
\end{eqnarray*}
Then $f_1, f_2$ are convenient and 
$$\Gamma(F)=\Gamma\{(3,0,0),(0,3,0),(0,0,5),(2,0,3),(0,2,3)\}.$$
Hence
$$\Gamma_\infty(F)=\{\Delta^1, \ldots,\Delta^{13}\},$$
with

\begin{enumerate}
\item[(i)] $\Delta^1=\Gamma\{(3,0,0),(0,3,0),(2,0,3),(0,2,3)\}
= \Gamma\{(2,0,0),(0,2,0)\}+\Gamma\{(1,0,0),(0,1,0),(0,0,3)\},$

\item[(ii)] $\Delta^2=\Gamma\{(2,0,3),(0,2,3),(0,0,5)\}=\Gamma\{(2,0,0),(0,2,0),(0,0,2)\}+\{(0,0,3)\},$

\item[(iii)] $\Delta^3=\Gamma\{(3,0,0),(0,3,0)\}=\Gamma\{(2,0,0),(0,2,0)\}+\Gamma\{(1,0,0),(0,1,0)\},$

\item[(iv)] $\Delta^4=\Gamma\{(2,0,3),(0,2,3)\}=\Gamma\{(2,0,0),(0,2,0)\}+\{(0,0,3)\},$

\item[(v)] $\Delta^5=\Gamma\{(2,0,3),(0,0,5)\}=\Gamma\{(2,0,0),(0,0,2)\}+\{(0,0,3)\},$

\item[(vi)] $\Delta^6=\Gamma\{(3,0,0),(2,0,3)\}=\{(2,0,0)\}+\Gamma\{(1,0,0),(0,0,3)\},$

\item[(vii)] $\Delta^7=\Gamma\{(0,2,3),(0,0,5)\}=\Gamma\{(0,2,0),(0,0,2)\}+\{(0,0,3)\},$

\item[(viii)] $\Delta^8=\Gamma\{(0,3,0),(0,2,3)\}=\{(0,2,0)\}+\Gamma\{(0,1,0),(0,0,3)\},$

\item[(ix)] $\Delta^9=\{(3,0,0)\}=\{(2,0,0)\}+\{(1,0,0)\},$

\item[(x)] $\Delta^{10}=\{(0,3,0)\}=\{(0,2,0)\}+\{(0,1,0)\},$

\item[(xi)] $\Delta^{11}=\{(2,0,3)\}=\{(2,0,0)\}+\{(0,0,3)\},$

\item[(xii)] $\Delta^{12}=\{(0,2,3)\}=\{(0,2,0)\}+\{(0,0,3)\},$

\item[(xiii)] $\Delta^{13}=\{(0,0,5)\}=\{(0,0,2)\}+\{(0,0,3)\},$
\end{enumerate}
By computation, it is not hard to show that the corresponding matrices $M_{\Delta^j}$ have rank $2$ on $(\mathbb{R}\setminus\{0\})^3$ for $j=1, \ldots ,13$. Hence $F$ is non-degenerate at infinity. By Theorem \ref{HolderTypeTheorem}, $F$ has a H\"{o}lder-type global error bound with the exponents $\alpha =\displaystyle\frac{2}{2d(12d - 3)^{n + p - 1}} =\frac{1}{ 3 \times 33^4}$ and $\beta = 1.$
}\end{example}

\begin{example}{\rm
Let
$$F=(f_1,f_2):\mathbb{R}^2\to\mathbb{R}^2,(x,y)\mapsto (x^2-y^2,x-y).$$
It is clear that
\begin{eqnarray*}
\Gamma(f_1) &=& \Gamma\{(2,0),(0,2\}, \\
\Gamma(f_2) &=&\Gamma\{(1,0),(0,1)\}, \\
\Gamma(F) &=& \Gamma(f_1) + \Gamma(f_2) = \Gamma\{(3,0),(0,3)\}.
\end{eqnarray*}

Consider the edge $\Delta := \Gamma\{(3,0),(0,3)\}=\Gamma\{(2,0),(0,2)\}+\Gamma\{(1,0),(0,1)\}$ of $\Gamma_\infty(F)$, then $F_\Delta=F$ and 
$$M_\Delta=
\begin{pmatrix}
2 x^2 & -2 y^2 & x^2-y^2 &  0\\
x     & -y     & 0       & x-y\\
\end{pmatrix}.
$$
It is clear that $\mathrm{rank}M_\Delta= 1$ when $x=y\not=0$, so the condition of non-degeneracy at infinity is not satisfied. However, by a small perturbation $F_\epsilon=F+(0,\epsilon x)=(x^2-y^2,x-y+\epsilon x)$, for any non-zero small $\epsilon$, we still have $\Gamma_\infty(F_\epsilon)=\Gamma_\infty(F)$ and the reader may check that the new polynomial map $F_\epsilon$ is non-degenerate at infinity. The exponents of the H\"{o}lder-type global error bound of $F_\epsilon$ will be $\alpha =\displaystyle\frac{2}{2d(12d - 3)^{n + p - 1}} =\frac{1}{ 2 \times 21^3}$ and $\beta = 1.$
}\end{example}

\section{Applications}\label{Applications}

In  this  section, we  describe various  applications of the  error bound results  obtained in the  previous  sections. % Some  of the  results  to  follow have sharpen existing error bounds. 

\subsection{$0-1$ integer feasibility problem}

Let $F \colon \Bbb R^n\rightarrow\Bbb R^p$ and $g \colon \Bbb R^n\rightarrow\Bbb R$ be polynomial maps. Consider the following $0-1$ integer feasibility problem (see \cite{Luo1994-2}):
$$F(x)\le 0, \quad g(x)=0,\quad x_i=0 \ \textrm{ or } \ 1, \ i=1,\ldots,n.$$
Equivalently, we may consider the following system:
$$F(x)\le 0, \quad g(x)=0,\quad x_i(x_i-1) = 0 , \ i=1,\ldots,n.$$
Assume that the solution set $S$ of the problem is not empty, and set 
$$r(x):= \|[f(x)]_+\| + |g(x)|+\displaystyle\sum_{i=1}^n|x_i(x_i-1)|.$$
It is clear that $r(x)$ is a nonnegative, continuous semi-algebraic function. Moreover, $r(x)$ is proper, i.e., $r(x^k)\to\infty$ for any sequence $x^k\to\infty$. Hence $r(x)$ satisfies the Palais-Smale condition at any level $t\ge 0$. By Theorem \ref{TheoremPalais-Smale-Global-Holderian}, there exist some constants $c > 0, \alpha > 0,$ and $\beta > 0$ such that for all $x \in \mathbb{R}^n$, we have
$$c d(x,S) \le [r(x)]^{\alpha} + [r(x)]^{\beta}.$$
In fact, this H\"older-type global error bound still holds if $F$ and $g$ are analytic mapping so this generalizes Theorem 5.6 in \cite{Luo1994-2}, which gives a H\"older-type error bound in the compact setting.

\subsection{Partition problem}

The partition problem asks whether an integer sequence $a_1 , \ldots, a_n$ can be partitioned, i.e., whether there exists $x \in \{\pm 1\}^n$ such that $\sum_{j = 1}^n a_j x_j = 0.$ This problem is known to be NP-complete (see \cite{Garey1979}). If  the infimum $f^*$ of the polynomial $f := (\sum_{j = 1}^n a_jx_j)^2 + \sum_{j = 1}^n(x_j^2-1)^2$ on ${\Bbb R}^n$ is equal to $0,$ a global minimizer is $\pm 1$-valued and thus provides a partition of the sequence.

We leave the reader verifying that $f$ is convenience and non-degenerate at infinity. It follows from Theorem~1.1 in \cite{Dinh2013} that $f$ attains its infimum on $\mathbb{R}^n.$ We have $d=4$ and $p=1$, so $\mathcal{H}(2d, n, p) = \mathcal{H}(8, n, 1) = 8(45)^n.$ Assume that $f^* = 0;$ then $S=\{x\in\Bbb R^n : \ f(x)\le 0\}=\{x\in\Bbb R^n : \ f(x)=0\}\not=\emptyset.$ By Theorem \ref{HolderTypeTheorem}, the following H\"older-type global error bound holds
$$c d(x,S) \le [f(x)]^{\frac{1}{8(45)^n}} + f(x) \quad \textrm{ for all } \quad x \in \mathbb{R}^n.$$

\subsection{Growth rate of the objective function in a polynomial optimization program}

Let $f_0$ and $f_1,\ldots,f_p:\Bbb{R}^n\to\Bbb{R}$ be polynomial functions in $n$ real variables. Set 
$$S:=\{x\in\Bbb{R}^n:\ f_1(x)\le 0,\ldots,f_1(x)\le 0\}.$$
Assume that $S$ is non empty. Let us consider the following constrained optimization problem
$$f^*:=\inf f_0(x) \ \text{ such that }\ x\in S$$
of minimizing $f_0$ over $S$. Under the condition of convenience and non-degeneracy at infinity, we have proved in \cite{Dinh2013} that if $f_0$ is bounded from below on $S$, then $f_0$ attains its infimum on $S$. Hence Theorem \ref{HolderTypeTheorem} can be applied to the solution set of this nonlinear polynomial program to obtain a growth property of the objective function.
\begin{corollary} Assume that $f_0$ is bounded from below on $S$ and that the map $(f_0,f_1,\ldots,f_p)$ is convenient and non-degenerate at infinity, then $f_0$ attains its infimum on $S$. Let 
$$A:=\{x\in S: f_0(x)=f^*\}\not=\emptyset$$
be the set of globally optimal solutions of $f$ on $S.$ Then there exists a constant $c>0$ such that 
$$c d(x, A) \le [f_0(x) - f^*]^\frac{2}{\mathcal{H}(2d, n, p)} + [f_0(x)- f^*]  \quad \textrm{ for all } \quad  x\in S.$$
\end{corollary}
%\begin{proof} Under the assumption of the corollary, by \cite{Dinh2013}, $f_0$ attains its infimum on $S$, hence $A\not=\emptyset$. The last conclusion of the corollary follows easily.
%\end{proof}

\subsection{Global H\"olderian stability for set-valued maps}
Let $F=(f_1,\ldots,f_p) \colon \Bbb{R}^n\to\Bbb{R}^p$ be a polynomial map. We define the set-valued map 
$S \colon \Bbb{R}^p \rightrightarrows \Bbb{R}^n$ by 
$$S(y):=\{x\in\Bbb{R}^n:\ f_i(x)-y_i\le 0, \ i=1,\ldots, p\} \quad \textrm{ for } y := (y_1, \ldots, y_p).$$
Then we have the following global H\"olderian property of the set-valued map $S.$

\begin{corollary} Assume that $F$ is convenient and non-degenerate at infinity. Then there exists a positive constant $c$ such that
$$S(y)\subseteq  S(0)+c\, (\|y\|^\frac{2}{\mathcal{H}(2d, n, p)}+\|y\|)\Bbb{B} \quad \textrm{ for all } \quad y \in \mathbb{R}^p,$$
where $\Bbb{B}$ denotes the closed unit Euclidean ball centered at the origin in $\Bbb{R}^n.$
\end{corollary}
\begin{proof} 
Let us define the function $f \colon \mathbb{R}^n \rightarrow \mathbb{R}$ by $f(x) := \max_{i = 1, \ldots, p} f_i(x).$ 
In view of Theorem~\ref{HolderTypeTheorem}, there exists a constant $c' > 0$ such that
\begin{eqnarray*}
c'd(x,S(0)) &\le & 
[f(x)]_+^\frac{2}{\mathcal{H}(2d, n, p)} + [f(x)]_+ \quad \textrm{ for all } \quad x \in \mathbb{R}^n.
\end{eqnarray*}

Let $y := (y_1, \ldots, y_p)$ be arbitrary in $\Bbb{R}^p.$ It suffices to show that
\begin{eqnarray*}
c'd(x,S(0)) & \le & \|y\|^\frac{2}{\mathcal{H}(2d, n, p)} + \|y\| \quad \textrm{ for all } \quad  x \in S(y).
\end{eqnarray*}
In fact, take any $x \in S(y).$ Then 
$$[f(x)]_+ = \max \{f(x), 0\} \le \max\{y_1, \ldots, y_p, 0\} \le \|y\|.$$
Therefore
\begin{eqnarray*}
c'd(x,S(0)) &\le & 
[f(x)]_+^\frac{2}{\mathcal{H}(2d, n, p)} + [f(x)]_+ \ \le \  \|y\|^\frac{2}{\mathcal{H}(2d, n, p)} + \|y\|,
\end{eqnarray*}
which completes the proof.
\end{proof}

\begin{acknowledgements}
This research was performed while the authors had been visiting at Vietnam Institute for Advanced Study in Mathematics (VIASM). The authors would like to thank the Institute for hospitality and support.

\thanks{$^\dagger$These authors were partially supported by Vietnam National Foundation for Science and Technology Development (NAFOSTED) grant 101.04-2014.23 and the Vietnam Academy of Science and Technology (VAST)}

\thanks{$^\ddagger$This author was partially supported by Vietnam National Foundation for Science and Technology Development (NAFOSTED) grant 101.04-2013.07}

\end{acknowledgements}

\end{document}